\documentclass{daj}
\usepackage{amsfonts}
\usepackage{amssymb}
\usepackage{amsmath, amsthm,color}
\usepackage{pdfsync}
\usepackage{bbm, dsfont}
\usepackage[margin=0.96in]{geometry}

\newtheorem{theorem}{Theorem}[section]
\newtheorem{corollary}[theorem]{Corollary}
\newtheorem{lemma}[theorem]{Lemma}
\newtheorem{prop}[theorem]{Proposition}

\dajAUTHORdetails{%
  title = {Square functions and  the Hamming cube: Duality}, 
  author = {Paata Ivanisvili, Fedor  Nazarov and  Alexander Volberg},
  plaintextauthor = {Paata Ivanisvili, Fedor  Nazarov, Alexander Volberg},
  plaintexttitle = {Square functions and  the Hamming cube: Duality}, 
  runningtitle = {Square functions and  the Hamming cube: Duality}, 
  runningauthor = {Paata Ivanisvili, Fedor  Nazarov, Alexander Volberg},
  copyrightauthor = {P.~Ivanisvili, F.~Nazarov and  A.~Volberg},
  keywords = {Square function, Hamming cube, Duality},
}   

\dajEDITORdetails{%
   year={2018},
   number={1},
   received={6 July 2017},   
   published={19 January 2018},  
   doi={10.19086/da.3113},       
}   

\begin{document}

\begin{frontmatter}[classification=text]

\title{Square functions and  the Hamming cube: Duality} 

\author[pgom]{Paata Ivanisvili\thanks{This paper is  based upon work supported by the National Science Foundation under Grant No. DMS-1440140 while two of the authors were in residence at the Mathematical Sciences Research Institute in Berkeley, California, during the Spring and Fall  2017 semester.}}
\author[johan]{Fedor Nazarov\thanks{Supported by NSF DMS-0800243}}
\author[laci]{Alexander Volberg\thanks{Supported by NSF DMS-1600065}}

\begin{abstract}
For $1<p\leq 2$,  any $n\geq 1$ and any  $f:\{-1,1\}^{n} \to \mathbb{R}$,  we obtain $(\mathbb{E} |\nabla f|^{p})^{1/p} \geq c(p)(\mathbb{E}|f|^{p} - |\mathbb{E}f|^{p})^{1/p}$ where $c(p)$ is the smallest positive zero of the confluent hypergeometric function ${}_{1}F_{1}(\frac{p}{2(1-p)}, \frac{1}{2}, \frac{x^{2}}{2})$.   Our approach is based on a certain duality  between the classical square function estimates on the Euclidean space and the gradient estimates on the Hamming cube.
\end{abstract}
\end{frontmatter}

\section{Main result}
Consider the Hamming cube $\{-1,1\}^{n}$ of an arbitrary dimension $n\geq 1$. For any $f :\{-1,1\}^{n} \to \mathbb{R}$ define the discrete partial derivative $\partial_{j} f(x)$ as follows
\begin{align*}
\partial_{j} f(x) = \frac{f(x)-f(S_{j}(x))}{2}, \quad x = (x_{1}, \ldots, x_{n}) \in \{-1,1\}^{n}, 
\end{align*}
where $S_{j}(x)$ is obtained from $x$ by changing the sign of $j$'th coordinate of $x$. Set $\nabla f(x) := (\partial_{1} f(x), \ldots, \partial_{n} f(x))$, and we define the norm of the discrete gradient
\begin{align*}
|\nabla f|^{2}(x) := \sum_{j=1}^{n} (\partial_{j} f(x))^{2} = \sum_{y \sim x} \left( \frac{f(x)-f(y)}{2}\right)^{2},
\end{align*}
where the summation in the last term runs over all neighbor vertices of $x$ in $\{-1,1\}^{n}$. Set 
\begin{align*}
\mathbb{E} f = \frac{1}{2^{n}} \sum_{x \in \{-1,1\}^{n}} f(x).
\end{align*}

\begin{theorem}\label{osnovnaya}
For any $1<p\leq 2$, $n\geq 1$, and any $f:\{-1,1\}^{n} \to \mathbb{R}$ we have
\begin{align}
s_{p'} (\mathbb{E} |f|^{p} - |\mathbb{E} f|^{p})^{1/p} \leq  (\mathbb{E} |\nabla f|^{p})^{1/p}. \label{nasha}
\end{align}
Here $p'=\frac{p}{p-1}$ is the conjugate exponent of $p$, and by $s_{q}$ we denote the smallest positive zero of the confluent hypergeometric function ${}_{1}F_{1}(-\frac{q}{2}, \frac{1}{2}, \frac{x^{2}}{2})$ (see (\ref{series}) for the definition). 
\end{theorem}

In Lemma~\ref{smalllest} we obtain a lower bound $s_{p'}\geq \sqrt{2/p'}$ for $1<p\leq 2$ which is precise when $p\to 2$. 
If $p'=2k$ for $k\in \mathbb{N}$, then $s_{p'}$ becomes the smallest positive zero of the Hermite polynomial $H_{2k}(x)$ where 
\begin{align*}
H_{m}(x) = \int_{\mathbb{R}}(x+iy)^{m} \frac{e^{-y^{2}/2}}{\sqrt{2\pi}} dy.
\end{align*}
The  constant  $s_{p'}$ in (\ref{nasha}) is larger then all previously known bounds~\cite{NAOR, Efraim} when $p$ is in a neighborhood of $2$, say $p\in (1.26, 2)$. For example, the estimate (\ref{nasha}) improves the Naor--Schechtman bound \cite{NAOR}  for the class of real valued functions for all $1<p<2$. Indeed, it follows from an application of Khinchin inequality with the sharp constant and (\ref{nasha}) that we have the following corollary 

\begin{corollary}\label{asebi}
For any $1<p\leq 2$, $n\geq 1$, and any $f : \{-1,1\}^{n} \to\mathbb{R}$ we have 
\begin{align}\label{assaf1}
s^{p}_{p'}  2^{\frac{p-2}{2}} \min\left(1,\frac{\Gamma((p+1)/2)}{\Gamma(3/2)} \right)  \mathbb{E}| f- \mathbb{E}f |^{p} \leq  \mathbb{E}_{x}\, \mathbb{E}_{x'} \left|\sum_{j=1}^{n} x'_{j}\partial_{j} f(x)\right |^{p},
\end{align}
where $\mathbb{E}_{x}$ and  $\mathbb{E}_{x'}$ average in variables $x$ and $x' =(x'_{1}, \ldots, x'_{n}) \in \{-1,1\}^{n}$ correspondingly. 
\end{corollary}
We will see in Proposition~\ref{comp} that 
\begin{align*}
s^{p}_{p'}  2^{\frac{p-2}{2}} \min\left(1,\frac{\Gamma((p+1)/2)}{\Gamma(3/2)} \right) > (p-1)^{p} \quad \text{for} \quad 1<p<2. 
\end{align*}
The latter implies that the estimate (\ref{assaf1}) improves the bound  of  Naor--Schechtman  for $1<p<2$ in the case of real valued functions (see Theorem~1 in \cite{NAOR} where $\beta_{p}(\mathbb{R})=1/(p-1)$).

\bigskip

On the other hand  $s_{p'}$ degenerates to $0$ when $p\to 1+$ which should not be the case for the best possible constant  by a result of Talagrand (see Section~\ref{sobolevskie}). For this endpoint case, when $p$ is close to $1$,  the result of Ben-Efraim--Lust-Piquard~\cite{Efraim} gives the better bounds  
\begin{align}\label{piquard1}
\frac{2}{\pi} (\mathbb{E}|f-\mathbb{E}f|^{p})^{1/p}\leq (\mathbb{E}|\nabla f|^{p})^{1/p} \quad 1\leq p\leq 2, 
\end{align}
and when $p=1$ it is widely believed that the sharp constant in the left hand side of (\ref{piquard1}) should be $\sqrt{2/\pi}$ instead of $2/\pi$ (see Section~\ref{sobolevskie} for more details). 

\bigskip 

We think that the main contribution of the current paper is not just Theorem~\ref{osnovnaya} that we obtain but rather a new {\em duality} approach that we develop between two different classes of extremal problems: square function estimates on the interval $[0,1]$  and gradient estimates on the Hamming cube, and Theorem~\ref{osnovnaya} should be considered as an example. Roughly speaking  one can take a valid estimate for a square function, dualize it by a certain double Legendre transform, and one can write its corresponding  dual estimate on the Hamming cube and vice versa. To illustrate another  example of our duality approach, in Section~\ref{poo} we present a short proof of the following theorem which improves a well--known inequality of Beckner
\begin{theorem}[see \cite{IVVO}]
For any $n\geq 1$, and any $f:\{-1,1\}^{n} \to \mathbb{R}$ we have 
\begin{align*}
\mathbb{E}\,  \Re\,  (f+i|\nabla f|)^{3/2} \leq (\Re\, \mathbb{E} f)^{3/2},
\end{align*}
where $\Re$ denotes the real part, and  $z^{3/2}$  is understood in the sense of principal brunch in the upper half-plane. 
\end{theorem}

\bigskip

  Going back to Theorem~\ref{osnovnaya},  it will be explained  later that $s_{p'}$  in a ``dual'' sense coincides with the sharp constant  found by B.~Davis in the $L^{q}$ norm estimates 
\begin{align}
&s_{q}  \| T^{1/2}\|_{q}\leq \|B_{T}\|_{q}, \quad q\geq 2, \quad \| T^{1/2}\|_{q}<\infty; \label{Davis01}\\
& \| B_{T}\|_{p}\leq s_{p}\|T^{1/2}\|_{p}, \quad 0<p\leq  2.  \label{Davis02}
\end{align}
Here $B_{t}$ is the standard Brownian motion starting at zero, and $T$ is any  stopping time. It was explained in \cite{Davis} that the same sharp estimates  (\ref{Davis01}) and (\ref{Davis02}) hold with $B_{T}$ replaced by an integrable function $g$ on $[0,1]$ with mean zero, and $T^{1/2}$ replaced by  the dyadic square function of $g$.

We notice the essential difference between the Davis estimates (\ref{Davis01}), (\ref{Davis02}) and (\ref{nasha}) that for a given power $p, 1<p\leq 2$, we need the ``dual'' constant  $s_{p'}=s_{\frac{p}{p-1}}$ in the theorem. Besides, inequality (\ref{nasha}) cannot be extended to the full range of exponents $p$ with some finite strictly positive constant $c(p)$ unlike (\ref{Davis01}) and (\ref{Davis02}) (see~\cite{Davis, BUR, BUR01} and (\ref{Latala})).

\section{Proof of the main result}

\subsection{An anonymous Bellman function}\label{anonymous}
In this section we want to define a function  $U: \mathbb{R}^{2} \to \mathbb{R}$ that satisfies  some special properties. 
Let $\alpha \geq  2$ and let  $\beta =\frac{\alpha}{\alpha-1}\leq  2$  be the conjugate exponent of $\alpha$. 
Let  
\begin{align}\label{series}
N_{\alpha }(x):={}_{1}F_{1}\left(-\frac{\alpha }{2}, \frac{1}{2}, \frac{x^{2}}{2} \right) = \sum_{m=0}^{\infty}\frac{(-2x^{2})^{m}}{(2m)!}\frac{\alpha }{2}\left(\frac{\alpha }{2}-1\right)\cdots \left(\frac{\alpha }{2}-m+1 \right)=1-x^{2} \frac{\alpha}{2}+...
\end{align}
be the confluent hypergeometric function.  $N_{\alpha }(x)$ satisfies the Hermite differential equation 
\begin{align}\label{hermit}
N''_{\alpha }(x)-xN'_{\alpha }(x)+\alpha N_{\alpha}(x)=0 \quad \text{for} \quad  x\in \mathbb{R}
\end{align}
with initial conditions $N_{\alpha}(0)=1$ and $N'_{\alpha}(0)=0$. Let $s_{\alpha}$ be the smallest positive zero of $N_{\alpha}$.

Set 
\begin{align*}
u_{\alpha}(x) := 
\begin{cases}
-\dfrac{\alpha s_{\alpha}^{\alpha-1}}{N'_{\alpha}(s_{\alpha})} N_{\alpha}(x), & 0\leq |x|\leq s_{\alpha};\\[10pt]
s_{\alpha}^{\alpha}-|x|^{\alpha}, & s_{\alpha} \leq |x|.
\end{cases}
\end{align*}
Clearly $u_{\alpha}(x)$ is $C^{1}(\mathbb{R}) \cap C^{2}(\mathbb{R}\setminus{\{s_{\alpha}\}})$ smooth even concave function.  The concavity follows from Lemma~\ref{root} and the fact that  $N'_{\alpha}(s_{\alpha}) <0$.  
Finally we define 
\begin{align}\label{upq}
U(p,q) := |q|^{\alpha} u_{\alpha}\left( \frac{p}{|q|}\right) \quad \text{with} \quad  U(p,0)=-|p|^{\alpha}.
\end{align}
For the first time the function $U(p,q)$ appeared in~\cite{Davis}. Later it was also used  in  \cite{Wang, Wang2} in the form  $\widetilde{u}(p,t)=U(p,\sqrt{t})$, $t\geq 0$. It was explained in~\cite{Davis} that $U(p,q)$ satisfies the following properties:
\begin{align}
&U(p,q) \geq |q|^{\alpha} s_{\alpha}^{\alpha} - |p|^{\alpha} \quad \text{for all} \quad \quad (p,q) \in \mathbb{R}^{2}, \quad \text{and when}\;  q=0, \; \text{the equality holds}; \label{obstacle}\\
& 2U(p,q) \geq U(p+a,\sqrt{a^{2}+q^{2}}) + U(p-a, \sqrt{a^{2}+q^{2}})  \quad \text{for all} \quad (p,q,a) \in \mathbb{R}^{3}.\label{neravenstvo}
\end{align}
We should refer to (\ref{obstacle}) as the {\em obstacle condition}, and to  (\ref{neravenstvo}) as the {\em main inequality}. We caution the reader that in \cite{Davis}
one may not find (\ref{neravenstvo}) written explicitly but one will find its infinitesimal form 
\begin{align}\label{infinitesimal}
\widetilde{u}_{t} + \frac{\widetilde{u}_{pp}}{2} \leq  0 \quad \text{for} \quad  \widetilde{u}(p,t) = U(p,\sqrt{t}), 
\end{align}
which follows from the main inequality by expanding it into Taylor's series with respect to $a$ near  $a=0$ and comparing the second order terms. Here  $\widetilde{u}_{pp}$ is defined everywhere except the curve  $|p/\sqrt{t}| = s_{\alpha}$ where $\widetilde{u}$ is only differentiable once.  

In fact,  the reverse implication also holds, i.e., one can derive (\ref{neravenstvo}) from (\ref{infinitesimal}) for this special $U$. This  was done in the PhD thesis of Wang \cite{Wang2} but we will present a short  proof in Section~\ref{heatt}, which partly follows the  Davis argument. Essentially the same argument also appeared later in  \cite{BM} in a slightly different setting. 

The function $U(p,q)$ is essential in obtaining the result in the Davis paper, namely it is used in the proof of (\ref{Davis01}),  and the argument goes as follows. First one shows that  
\begin{align*}
X_{t} = U(B_{t},\sqrt{t}) \quad \text{for} \quad t \geq 0
\end{align*} 
is a supermartingale which is guaranteed by (\ref{infinitesimal}). Finally, by the optional stopping theorem,
\begin{align*}
\mathbb{E}(T^{\frac{\alpha}{2}}s_{\alpha}^{\alpha} - |B_{T}|^{\alpha}) \stackrel{(\ref{obstacle})}{\leq} \mathbb{E}U(B_{T}, \sqrt{T}) \leq U(0,0)=0,
\end{align*}
which yields  (\ref{Davis01}). One may notice that $U(p,q)$ is the minimal function with properties (\ref{obstacle}) and (\ref{neravenstvo}). 

Davis mentions that the proof presented in his paper was suggested by an anonymous referee, and this explains the title of the current section.

\subsection{Dualizing the  Bellman function $U(p,q)$ and  going to the Hamming cube}

Set $\Psi(p,q,x,y):=px+qy+U(p,q)$ for $x \in \mathbb{R}$ and $y \geq 0$. 
We define 
\begin{align}\label{dualbellman}
M(x,y) = \min_{q \leq 0} \sup_{p \in \mathbb{R}} \Psi(p,q,x,y) \quad \text{for} \quad  x \in \mathbb{R}, \; y \geq 0. 
\end{align}
\begin{lemma}\label{milion}
For each $(x,y)\in \mathbb{R}\times \mathbb{R}_{+}$, there exists $(p^{*},q^{*}) = (p^{*}(x,y), q^{*}(x,y))$ such that   
\begin{align}\label{fedja0}
\min_{q \leq 0} \sup_{p \in \mathbb{R}} \Psi(p,q,x,y)=\max_{p \in \mathbb{R}} \inf_{q \leq 0} \Psi(p,q,x,y)=\Psi(p^{*},q^{*},x,y)
\end{align}  
and we have  
\begin{align}\label{sedlo}
\Psi(p,q^{*},x,y) \leq \Psi(p^{*},q^{*},x,y) \leq \Psi(p^{*},q,x,y)\quad  \text{for all} \quad (p,q) \in \mathbb{R}\times \mathbb{R}_{-}. 
\end{align}
\end{lemma}
\begin{proof}

First let us show that for each fixed $(x,y)$ the function $\Psi(p,q,x,y)$ is convex in $q$ and concave in $p$.  The concavity  in $p$ follows from Lemma~\ref{root}, and the fact that $U$ is even and  $C^{1}$ smooth in $p$. 

To verify  the convexity in  $q$, it is enough to show that the map $q \mapsto U(p,q)$ is convex for $|p|\leq  |q| s_{\alpha} $.  Set $z=\frac{|p|}{|q|}  \in [0, s_{\alpha}]$. Then  we have 
\begin{align*}
&U_{qq} =|q|^{\alpha-2}\left[\alpha(\alpha-1)u_{\alpha}(z) - 2(\alpha-1)z u'_{\alpha}(z) +z^{2}u''_{\alpha}(z)\right] \stackrel{(\ref{hermit})}{=}\\
&|q|^{\alpha-2}\left[-(\alpha-1)z u'_{\alpha}(z) + (z^{2}-\alpha+1)u''_{\alpha}(z) \right].
\end{align*}
Since $u_{\alpha}(z)$ coincides with $N_{\alpha}(z)$ up to a positive constant, the convexity follows from Lemma~\ref{root} and the fact that $\alpha\geq 2$.

Notice that for each $(x,y) \in \mathbb{R}\times \mathbb{R}_{+}$ the map 
\begin{align}\label{fedja}
(p,q) \mapsto px + qy + |q|^{\alpha} u_{\alpha}\left(\frac{p}{|q|}\right)
\end{align}
 satisfies the assumptions of Theorem~\ref{hartung}  where we take $(p_{0}, q_{0})=(0,0)$ (see Section~\ref{minnn} in Appendix). Therefore the conclusions of Lemma~\ref{milion} follow from Theorem~\ref{hartung}. 
 \end{proof}
 


\begin{lemma}\label{doska}
For $\beta = \frac{\alpha}{\alpha-1}$,  any $x,a,b \in \mathbb{R}$, and any $y\geq 0$  we have 
\begin{align}
&M(x,y) \geq \left( \frac{\alpha-1}{\alpha^{\beta}}\right)\left(|x|^{\beta} -\frac{y^{\beta}}{s_{\alpha}^{\beta}}\right) \quad \text{and when} \; y=0 \; \text{the equality holds}; \label{obs2}\\
&2 M(x,y) \geq M(x+a, \sqrt{a^{2}+(y+b)^{2}})+M(x-a, \sqrt{a^{2}+(y-b)^{2}}) \,.\label{dualinequality}
\end{align}
\end{lemma}
The reader notices that   dualization (\ref{dualbellman}) produces   inequality (\ref{dualinequality}) that is  different from (\ref{neravenstvo}).
\begin{proof}
Set 
\begin{align*}
(x_{\pm}, y_{\pm}) := (x\pm a, \sqrt{a^{2}+(y\pm b)^{2}}).
\end{align*}

Lemma~\ref{milion} gives  points  $(p^{*},q^{*})$ and $(p^{\pm}, q^{\pm})$ corresponding to  $(x,y)$ and  $(x_{\pm}, y_{\pm})$. It follows from (\ref{sedlo}) that to prove (\ref{dualinequality}) it would be enough to find numbers  $p \in \mathbb{R}$ and $q_{1}, q_{2} \leq 0$ such that 
\begin{align*}
2\Psi(p,q^{*},x,y) \geq \Psi(p^{+}, q_{1}, x_{+}, y_{+})+ \Psi(p^{-}, q_{2}, x_{-}, y_{-}) .
\end{align*}

The right choice will be 
\begin{align}\label{choose}
p = \frac{p^{+}+p^{-}}{2}\quad \text{and} \quad  q_{1}=q_{2} = - \sqrt{\left(\frac{p^{+}-p^{-}}{2}\right)^{2} + (q^{*})^{2}}\,,
\end{align}
but let us explain it in details. 

Notice that  by Cauchy--Schwarz we have 
\begin{align*}
q_{1}\sqrt{a^{2}+(y+b)^{2}} + q_{2}\sqrt{a^{2}+(y-b)^{2}} - 2q^{*}y  \leq -|a|\left(\sqrt{q_{1}^{2}-(q^{*})^{2}}+\sqrt{q_{2}^{2}-(q^{*})^{2}} \right)
\end{align*}
provided that $q_{1}, q_{2} \leq q^{*}\leq  0$.  Indeed, we have 
\begin{align*}
&q_{1}\sqrt{a^{2}+(y+b)^{2}} + q_{2}\sqrt{a^{2}+(y-b)^{2}} - 2q^{*}y =\\
&-\sqrt{(q_{1}^{2}-(q^{*})^{2})+(q^{*})^{2}}\, \sqrt{a^{2}+(y+b)^{2}}-\sqrt{(q_{2}^{2}-(q^{*})^{2})+(q^{*})^{2}}\,  \sqrt{a^{2}+(y-b)^{2}} -2q^{*}y \leq \\
&-|a|\sqrt{q_{1}^{2}-(q^{*})^{2}}  - |q^{*}(y+b)|-|a|\sqrt{q_{2}^{2}-(q^{*})^{2}}  - |q^{*}(y-b)|-2q^{*}y\leq \\
&-|a|\left( \sqrt{q_{1}^{2}-(q^{*})^{2}}+\sqrt{q_{2}^{2}-(q^{*})^{2}}\right).
\end{align*}

Denoting $r_{j}^{2}=q_{j}^{2}-(q^{*})^{2}$ for $j=1,2$, we see that it is enough to find $p \in \mathbb{R}$ and $r_{1}, r_{2} \geq 0$ such that 
\begin{align*}
2(px+U(p,q^{*})) \geq  -|a|(r_{1}+r_{2})+p^{+}x_{+}+U\left( p^{+},\sqrt{r_{1}^{2}+(q^{*})^{2}}\right)+p^{-}x_{-}+U\left( p^{-},\sqrt{r_{2}^{2}+(q^{*})^{2}}\right).
\end{align*} 
By choosing  $p=\frac{p^{+}+p^{-}}{2}$, and substituting the values for $x_{\pm}=x\pm a$ we see that it would suffice to find $r_{1}, r_{2}\geq 0$ such that 

\begin{align*}
2U\left(\frac{p^{+}+p^{-}}{2}, q^{*} \right)\geq  -|a|(r_{1}+r_{2})+a(p^{+}-p^{-}) +U\left( p^{+},  \sqrt{r_{1}^{2}+(q^{*})^{2}}\right)+U\left( p^{-},  \sqrt{r_{2}^{2}+(q^{*})^{2}}\right)\,.
\end{align*}

 We will choose $r_{1}=r_{2} = \frac{|p^{+}-p^{-}|}{2}$. It follows from $-|a||p^{+}-p^{-}|+a(p^{+}-p^{-})\leq 0$ that we only need  to  have  the inequality 
 \begin{align*}
2 U\left(\frac{p^{+}+p^{-}}{2}, q^{*}\right) \geq  U\left(p^{+}, \sqrt{\left( \frac{p^{+}-p^{-}}{2}\right)^{2}+(q^{*})^{2}}\right) + U\left(p^{-}, \sqrt{\left( \frac{p^{+}-p^{-}}{2}\right)^{2}+(q^{*})^{2}}\right).
\end{align*}
But this inequality follows from  (\ref{neravenstvo}). 

\medskip

To verify the obstacle condition (\ref{obs2}), notice that  (\ref{obstacle}) for $U(p,q)$ gives 
\begin{align}\label{erti}
M(x,y) \geq \min_{q\leq 0}  \sup_{p}  \left(px + qy+|q|^{\alpha}s_{\alpha}^{\alpha}-|p|^{\alpha} \right)= \left( \frac{\alpha-1}{\alpha^{\beta}}\right)\left( |x|^{\beta} - \frac{y^{\beta}}{s_{\alpha}^{\beta}} \right).
\end{align}

Finally if $y=0$, then we obtain 
\begin{align*}
M(x,0) =  \max_{p} \inf_{q\leq 0} (px+U(p,q)) \stackrel{(*)}{=} \max_{p} (px+U(p,0))=\sup_{p} (px-|p|^{\alpha}) = \left( \frac{\alpha-1}{\alpha^{\beta}}\right) |x|^{\beta}\,.
\end{align*}
Equality (*) follows from the fact that  
\begin{align*}
q \mapsto px + U(p,q)
\end{align*}
is an even convex map. 

\end{proof}
\begin{corollary}\label{selfimpr}
For any $a, x \in \mathbb{R}$, all $y, b \in \mathbb{R}^{N}$,  and any $N\geq 1$, we have 
\begin{align}\label{main}
\frac{1}{2}\left(M(x+a, \sqrt{a^{2}+\|y+b\|^{2}})  +M(x-a, \sqrt{a^{2}+\|y-b\|^{2}})\right)\leq M(x, \|y\|).
\end{align}
\end{corollary}
\begin{proof}
It follows from the definition of $M$ that the map $y \mapsto M(x,y)$ is decreasing in $y$ for $y\geq 0$. Therefore by (\ref{dualinequality}) and the triangle inequality we obtain 
\begin{align*}
&\frac{1}{2}\left(M(x+a, \sqrt{a^{2}+\|y+b\|^{2}})  +M(x-a, \sqrt{a^{2}+\|y-b\|^{2}})\right) \leq \\
&M\left(x, \frac{\|y+b\|+\|y-b\|}{2} \right) \leq M(x,\|y\|).
\end{align*}
\end{proof}

The inequality (\ref{main}) gives rise to the estimate 
\begin{align}\label{IV}
\mathbb{E} M(f, |\nabla f|)\leq M(\mathbb{E} f, 0) \quad \text{for all} \quad  f : \{-1,1\}^{n} \to \mathbb{R}. 
\end{align}

Indeed, the reader can find in~\cite{IVVO} the passage from (\ref{main}) to (\ref{IV}). In fact, inequality (\ref{main}) is the same as 
\begin{align}\label{inductive}
\mathbb{E}_{x_{j}} M(f, |\nabla f|) \leq M(\mathbb{E}_{x_{j}}f, |\nabla \mathbb{E}_{x_{j}} f|) \quad \text{for any} \quad f :\{-1,1\}^{n} \to \mathbb{R},
\end{align}
where $\mathbb{E}_{x_{j}}$ takes the average in the coordinate $x_{j}$, i.e., 
\begin{align*}
\mathbb{E}_{x_{j}} f = \frac{1}{2}\left( f\underbrace{(x_{1}, \ldots, 1, \ldots, x_{n})}_{\text{set $1$ on the $j$-th place}}+ f\underbrace{(x_{1}, \ldots, -1, \ldots, x_{n})}_{\text{set $-1$ on the $j$-th place}}\right).
\end{align*}
The rest follows by iterating (\ref{inductive}), the fact  that $\mathbb{E} = \mathbb{E}_{x_{1}}\ldots \mathbb{E}_{x_{n}}$ and $|\nabla \mathbb{E} f | =0$.

\subsection{The proof of Theorem~\ref{osnovnaya}}

We have 
\begin{align*}
\left( \frac{\alpha-1}{\alpha^{\beta}}\right) \mathbb{E}  \left( |f|^{\beta} - \frac{|\nabla f| ^{\beta}}{s_{\alpha}^{\beta}} \right) \stackrel{(\ref{obs2})}{\leq} \mathbb{E} M(f, |\nabla f|)\stackrel{(\ref{IV})}{\leq} M(\mathbb{E} f, 0) \stackrel{(\ref{obs2})}{=}\left( \frac{\alpha-1}{\alpha^{\beta}}\right) |\mathbb{E} f|^{\beta},
\end{align*}
and this gives inequality (\ref{nasha}).


\section{Remarks and Applications}

\subsection{Going from $U$ to $M$: from Square function to the Hamming cube}

Let $g$ be an integrable function on $[0,1]$. Let $D([0,1])$ denote all dyadic intervals in $[0,1]$. Consider the dyadic martingale $g_{n}$ defined as follows 
\begin{align}\label{ddy}
g_{n}(x) = \sum_{|I|=2^{-n},\, I \in D([0,1])} \langle g \rangle_{I} \mathbbm{1}_{I}(x),
\end{align}
where $\langle g \rangle_{I} = \frac{1}{|I|} \int_{I} g$. The square function $S(g)$ is defined as follows 
\begin{align*}
S(g)(x) = \left(\sum_{n=0}^{\infty} (g_{n+1}(x)-g_{n}(x))^{2}\right)^{1/2}.
\end{align*}
For convenience we always assume that the number of nonzero terms in (\ref{ddy}) is finite so that $S(g)(x)$ makes sense. 
Let $O(p,q)$ be a continuous real valued function, and suppose one wants to estimate the quantity   $\int_{0}^{1} O(g, S(g))$ from above in terms of $\int_{0}^{1}g$.  If one finds a function 
\begin{align}
&U(p,q)\geq O(p,q), \label{opa01}\\
&2U(p,q) \geq U(p+a,\sqrt{a^{2}+q^{2}}) + U(p-a, \sqrt{a^{2}+q^{2}}), \label{opa02}
\end{align}
 then one  obtains (see \cite{Wang}) the bound 
\begin{align*}
\int_{0}^{1} O(g,S(g)) \leq \int_{0}^{1} U(g, S(g)) \leq U\left(\int_{0}^{1} g, 0 \right).
\end{align*}
Conversely, suppose that the  inequality 
\begin{align}\label{hudson}
\int_{0}^{1}O(g, S(g)) \leq F\left( \int_{0}^{1} g\right)
\end{align}
holds for all integrable functions $g$ on $[0,1]$ and some $F$. Then there exists $U(p,q)$ such that the conditions (\ref{opa01}), (\ref{opa02}) are satisfied and $U(p,0)\leq F(p)$. 
Indeed, consider the extremal problem 
\begin{align}\label{ex1}
U(p,q) = \sup_{g} \left\{\int_{0}^{1} O(g, \sqrt{S(g)^{2}+q^{2}}), \quad \int_{0}^{1}g =p \right\}. 
\end{align}
This $U$ satisfies (\ref{opa01}) (take $g = p$ constant), and, in fact, it  satisfies  (\ref{opa02}). The latter fact can be proved by using the standard Bellman principle (see Chapter 8, \cite{OSE}, and survey ~\cite{NTV}). Besides,  
\begin{align*}
U(p, 0)  =  \sup_{g} \left\{\int_{0}^{1} O(g,S(g)), \quad \int_{0}^{1}g =p \right\} \leq  F(p)
\end{align*}
 because of  (\ref{ex1}). Therefore there is one to one correspondence between the extremal problems for the square function of the form (\ref{ex1}) and the functions $U(p,q)$ with the properties (\ref{opa01}) and (\ref{opa02}). 

The  gradient estimates on the Hamming cube are more subtle. Take any real valued  $\widetilde{O}(x,y)$ and suppose that we want to estimate $\mathbb{E} \widetilde{O}(f, |\nabla f|)$ from above in terms of $\mathbb{E} f$ for any $f : \{-1,1\}^{n} \to \mathbb{R}$ and for all $n \geq 1$. If one finds $M(x,y)$ such that 
\begin{align}
&M(x,y)\geq \widetilde{O}(x,y), \label{opa21}\\
&2 M(x,y) \geq M(x+a, \sqrt{a^{2}+(y+b)^{2}})+M(x-a, \sqrt{a^{2}+(y-b)^{2}}), \label{opa22}
\end{align}
 then\footnote{We do also need to assume that  $y \mapsto M(x,y)$ is decreasing in $y$ for each fixed $x$ to ensure Corollary~\ref{selfimpr}. But if $M$ is $C^{1}$ smooth  then $M_{y}\leq 0$ is guaranteed by (\ref{opa22}). Indeed,  if we take $a=0$ in (\ref{opa22}) we obtain that $y \mapsto M(x,y)$ is concave  for each $x$. Next, taking $y=b=0$ and sending $a \to 0+$, we obtain by Taylor's formula that $M_{y}(x,0)\leq 0$.  Therefore  $M_{y}(x,y)\leq 0$.} one can obtain the estimate (see~\cite{IVVO}) 
\begin{align}\label{PA}
\mathbb{E} \widetilde{O}(f, |\nabla f|)\leq \mathbb{E} M(f, |\nabla f|) \leq M(\mathbb{E} f, 0)\,.
\end{align}

Thus finding such $M$ is sufficient to obtain the estimate but it is unclear whether conditions (\ref{opa21}) and (\ref{opa22}) are \textbf{necessary} to obtain the bound $\mathbb{E} \widetilde{O}(f, |\nabla f|)\leq M(\mathbb{E} f, 0)$. In other words we do not know what is the corresponding extremal problem for $M$, i.e., what is the right Bellman function $M$. The reason lies in the fact that  there is an essential difference between the Hamming cube and the dyadic intervals, i.e., test functions do not concatenate in a good way on $\{-1,1\}^{n}$ as it happens for dyadic martingales. 

Now we formulate an abstract theorem that formalizes our duality principle  in a general setting. 
\begin{theorem}
Let $I, J \subseteq \mathbb{R}$ be convex sets. Take an arbitrary  $O(p,q) \in C(I\times\mathbb{R}_{+})$, and let $U (p,q) : I\times \mathbb{R}_{+} \to \mathbb{R}$ satisfy properties (\ref{opa01}) and (\ref{opa02}). Assume that for each $(x,y) \in J\times \mathbb{R}_{+}$, we have 
\begin{align*}
\min_{q\leq 0} \sup_{p\in I}\,  px+qy+U(p,|q|) = \max_{p \in I} \inf_{q\leq 0}\,  px+qy+U(p,|q|). 
\end{align*}
 Then  $M$ and $\widetilde{O}$ defined as  
\begin{align}
&M(x,y) = \min_{q\leq 0} \sup_{p \in I} \; \left( px +qy + U(p,|q|) \right), \label{realdual}\\
&\widetilde{O}(x,y) = \inf_{q\leq 0} \sup_{p \in I} \; \left( px +qy + O(p,|q|) \right) \label{oobs}
\end{align}
satisfy (\ref{opa22}) and (\ref{opa21}), and, thereby, (\ref{PA}) for any $f :\{-1,1\}^{n} \to J$ and any $n \geq 1$.   
\begin{align*}
\end{align*}
\end{theorem} 
One may think that finding $U(p,q)$ with the property (\ref{opa02}) is a difficult problem. Let us make a quick remark here that if it happens that  $t \mapsto U(p,\sqrt{t})$ is convex for each fixed $p\in I$, then (\ref{opa02}) is automatically implied by its infinitesimal form, i.e., by $U_{pp}+U_{q}/q\leq 0$ (see the proof of Lemma~\ref{barko}).
\begin{proof}
The proof essentially repeats the proof of Lemma~\ref{doska}. Let us sketch the argument. Define $\Psi(p,q,x,y):=px+qy+U(p,|q|)$. The existence of a {\em saddle} point $(p^{*}, q^{*})$ with properties  (\ref{fedja0}) and (\ref{sedlo}) is guaranteed by Lemma~\ref{magari2}. The convexity of the set $I$ allows us to choose $p$ from $I$, and $q_{1}, q_{2} \in (-\infty, 0]$ according to (\ref{choose}). The rest of the proof of the theorem is the same as in Lemma~\ref{doska}. Inequality (\ref{oobs}) follows from  (\ref{opa01}). Convexity of $J$ is needed, for example, to ensure that if $f : \{-1,1\}^{n} \to J$,  then $\mathbb{E} f \in J$, so that (\ref{PA}) makes sense. 
\end{proof}

\subsection{Going from $M$ to $U$: from Hamming cube to square function}
Another interesting observation is that equality (\ref{realdual}) was lurking in a solution of a certain Monge--Amp\`ere equation. For example, taking $a,b \to 0$ in  (\ref{opa22}),  and using the Taylor's series expansion (assuming that $M$ is smooth enough) one obtains 
\begin{align}\label{matrica}
\begin{pmatrix}
M_{xx}+\frac{M_{y}}{y} & M_{xy}\\
M_{xy} & M_{yy}
\end{pmatrix} \leq 0\,.
\end{align}
When  looking for the least function $M$ with $M \geq \widetilde{O}$ and (\ref{matrica}), it is reasonable to assume that  condition (\ref{matrica}) should degenerate except, possibly, on the set where $M$ coincides with its obstacle $\widetilde{O}$. The degeneracy of  (\ref{matrica})  means that the determinant of the matrix in  (\ref{matrica}) is zero. This is a general Monge--Amp\`ere type equation and,  after an application of the exterior differential systems of Bryant--Griffiths (see \cite{IVVO1}), we obtain that the solutions can be locally  characterized as follows: 
\begin{align}
&x=-U_{p}(p,q), \nonumber\\
&y=-U_{q}(p,q), \nonumber\\
&M(x,y)=px+qy+U(p,q), \label{bryant}
\end{align}
where $U$ satisfies the equation 
\begin{align}\label{heat}
U_{pp}+\frac{U_{q}}{q}=0.
\end{align}
In~\cite{IVVO1} we used $u(p,t)=-U(p,\sqrt{2t})$ instead of $U(p,q)$, in which case (\ref{heat}) becomes just the backward heat equation for $u(p,t)$. 
 We will not formulate a  formal statement  but we do make a remark that such a  reasoning allows us to guess the {\em dual} of $M$, i.e., to find $U$ given $M$.  The way  this guess works will be illustrated in Section~\ref{poo}.

\medskip

Our final remark is that one may try to use  $U(p,q) :=M(p,q)$ with $O(p,q):=\widetilde{O}(p,q)$ because (\ref{opa22}) clearly implies (\ref{opa02}). It will definitely give some estimate for the square function but not the sharp one. Indeed, for the sharp estimates,   condition (\ref{opa02}) for $U$ usually degenerates, namely (\ref{heat}) holds. On the other hand, if $M_{xx}+M_{y}/y=0$ and (\ref{matrica}) holds, then $M_{xy}=0$, and  
\begin{align}\label{bzik}
M(x,y)=C(x^{2}-y^{2})+Dx+Q
\end{align}
 for some constants $C, D, Q\in \mathbb{R}$. This family of functions corresponds to the trivial inequality $\int_{0}^{1}S(g)^{2} \leq \int_{0}^{1} g^{2}$.
  Analogously, the best possible function $U$ satisfying  (\ref{opa01}) and (\ref{opa02}) will almost never satisfy (\ref{matrica}) except for a very particular case when $U(p,q)=C(p^{2}-q^{2})+Dp+Q$.  

\subsection{The dual to Log-Sobolev is Chang--Wilson--Wolff}
The function $M(x,y)=x\ln x - \frac{y^{2}}{2x}$ satisfies (\ref{matrica}) and, therefore, it  gives the log-Sobolev inequality~\cite{IVVO1}. Its dual in the sense of (\ref{bryant}) is $U(p,q)=e^{p-q^{2}/2}$ (see Section 3.1.1 in~\cite{IVVO1} where $t=q^{2}/2$). Notice that for this  $U$, inequality  (\ref{opa02}) simplifies to  
\begin{align*}
2e^{a^{2}/2}\geq e^{a}+e^{-a},
\end{align*}
which is true since $(2k)! \geq 2^{k} k!$ for $k \geq 0$. Therefore we obtain 
\begin{corollary}
For any integrable $g$ on $[0,1]$, we have
\begin{align*}
\int_{0}^{1} \exp\left( g-\frac{S^{2}(g)}{2}\right) \leq \exp\left( \int_{0}^{1} g\right).
\end{align*}
\end{corollary}
This corollary immediately recovers the result of Chang-Wilson-Wolf~\cite{CWW} well-known to probabilists, namely for any $g$ with $\int_{0}^{1} g=0$ and $\|S(g)\|_{\infty}<\infty$, we have 
\begin{align}\label{suparupa}
\int_{0}^{1} e^{g} \leq e^{\|S(g)\|^{2}_{\infty}/2}.
\end{align}
Next, repeating a standard argument, namely, considering $tg$ and applying Chebyshev's inequality (see Theorem~3.1 in \cite{CWW}), one obtains the superexponential bound 
\begin{align}\label{Wolff}
| \{ x \in [0,1]\, :\,  g(x) - \int_{0}^{1}g \geq \lambda \} | \leq e^{-\frac{1}{2} \lambda^{2}/\|Sg\|_{\infty}^{2}}
\end{align}
for any $\lambda \geq 0$.

We should remind that the log-Sobolev inequality via the Herbst argument~\cite{Ledoux} gives Gaussian concentration inequalities, namely, 
\begin{align}\label{sup2}
\gamma\left(x \in \mathbb{R}^{n}\, : f(x)-\int_{\mathbb{R}^{n}} f d\gamma \geq \lambda  \right) \leq e^{-\frac{1}{2} \lambda^{2}/\|\nabla f\|_{\infty}^{2}}
\end{align}
for any $\lambda \geq 0$ and any smooth  $f: \mathbb{R}^{n} \to \mathbb{R}$ with $\|\nabla f\|_{\infty}<\infty$. Here  $\gamma$ is the standard Gaussian measure on $\mathbb{R}^{n}$. 

In other words we just illustrated that estimates (\ref{sup2}) and (\ref{Wolff}) are dual to each other in the sense of duality between functions $M=x\ln x - \frac{y^{2}}{2x}$ and $U=e^{p-q^{2}/2}$. 

\subsection{Poincar\'e inequality 3/2: a simple proof via duality}\label{poo}
It was proved in \cite{IVVO} that for any $f :\{-1,1\}^{n} \to \mathbb{R}$, we have 
\begin{align}\label{poincare}
\mathbb{E}\,  \Re \, (f+i|\nabla f|)^{3/2} \leq \Re (\mathbb{E} f)^{3/2},
\end{align}
where $z^{3/2}$ is taken in the sense of the principal brunch in the upper half-plane.  Inequality (\ref{poincare}) improves Beckner's bound for a particular exponent~\cite{IVVO}. Consider  
\begin{align*}
M(x,y) = \Re (x+iy)^{3/2} = \frac{1}{\sqrt{2}} (2x - \sqrt{x^{2}+y^{2}}) \sqrt{\sqrt{x^{2}+y^{2}}+x} \quad \text{for} \quad  (x,y)\in \mathbb{R}^{2}.
\end{align*}
It was explained in \cite{IVVO} that to prove (\ref{poincare}) it is enough to check that $M(x,y)$ satisfies (\ref{opa22}), and the latter fact involved careful investigation of the roots of several  very high degree polynomials with integer coefficients. Let us give a simple proof of (\ref{opa22}) using our duality technique. 
\begin{prop}
The function $M(x,y)= \Re (x+iy)^{3/2}$ satisfies (\ref{opa22}) for all $x,a,b \in \mathbb{R}$ and $y\geq 0$. 
\end{prop}
\begin{proof}
 $M(x,y)$ is a solution of the  homogeneous Monge--Amp\`ere equation (\ref{matrica}), and therefore it has a representation of the form (\ref{bryant}) (see Section~3.1.4 in \cite{IVVO1}):
\begin{align*}
&x=-U_{p}(p,q);\\
&y=-U_{q}(p,q);\\
&U(p,q)=- \frac{4}{27}(p^{3}-3pq^{2});\\
&M(x,y)=px+qy+U(p,q).
\end{align*}
This leads us to the following guess
\begin{align*}
\frac{1}{\sqrt{2}} (2x - \sqrt{x^{2}+y^{2}}) \sqrt{\sqrt{x^{2}+y^{2}}+x}=\min_{q\leq 0} \sup_{p\geq 0}\left(  xp + qy - \frac{4}{27}(p^{3}-3pq^{2}) \right),
\end{align*}
which can be directly  checked.  Using Theorem~\ref{hartung} with $(p_{0},q_{0})=(0,0)$, and following the proof of Lemma~\ref{doska}, it is enough to check that $U(p,q)$ satisfies (\ref{opa02}). Notice that (\ref{opa02}) is an identity for $U(p,q)= -\frac{4}{27}(p^{3}-3pq^{2})$. This finishes the proof of the proposition. 
\end{proof}

\subsection{Sobolev inequalities}\label{sobolevskie}
\subsubsection{The Hamming cube $\{-1,1\}^{n}$}
For $p \in [1,2]$, let  $c_{p}$ be the best possible constant such that  
\begin{align}\label{sobham}
c_{p}(\mathbb{E}|f|^{p} - |\mathbb{E} f|^{p}) \leq \mathbb{E} |\nabla f|^{p} \quad \text{for all functions} \quad  f:\{-1,1\}^{n} \to \mathbb{R}.
\end{align}
Our theorem implies that $c_{p} \geq s^{p}_{p'}$ for $p\in (1,2]$.  Notice that when $p=2$, we have $c_{2}=s_{2}^{2}=1$, and (\ref{sobham}) recovers the classical Poincar\'e inequality. When $p \to 1+$ the constant $s^{p}_{p'}$ tends to zero which should not be the case for $c_{p}$. Indeed, it follows from a deep result of Talagrand~\cite{Talagrand} that if $T_{p}$ is the best possible constant in the following estimate
 \begin{align}\label{Ramonn}
 T_{p}\,  \mathbb{E} |f-\mathbb{E} f|^{p} \leq \mathbb{E} |\nabla f|^{p}  \quad \text{for all} \quad  f:\{-1,1\}^{n} \to \mathbb{R},
 \end{align}
then $T_{p}>0$ for all $p \in [1, \infty)$. Now notice that $T_{1}=c_{1}$, $T_{2}=c_{2}$ and $T_{p} \geq c_{p}$ for $p\in (1, 2)$. When $p>2$,  by example (\ref{Latala}), we must have $c_{p}=0$ unlike the fact that $T_{p}>0$ for $p>2$. So one may wonder whether the  positivity of $T_{p}$  may not imply the positivity of $c_{p}$ on the interval $(1,2)$. Let us mention that this is not the case, in fact $2c_{p} \geq T_{p}$ for $p\in (1,2)$. Indeed,  it will suffice to prove that  
$2\mathbb{E} |f-\mathbb{E}f|^{p} \geq \mathbb{E}|f|^{p} - |\mathbb{E} f|^{p}$. If $\mathbb{E}f=0$ this is obvious. Assume $\mathbb{E}f\neq 0$. Next, we show a simple inequality 
\begin{align}\label{ramon}
2|x-1|^{p}-|x|^{p}+1 \geq p(1-x) \quad \text{for all} \quad1\leq p\leq 2 \quad \text{and} \quad   x \in \mathbb{R}. 
\end{align}
Plugging $x= f/\mathbb{E} f$,  and taking the expectation,  we obtain $2\mathbb{E} |f-\mathbb{E}f|^{p} \geq \mathbb{E}|f|^{p} - |\mathbb{E} f|^{p}$.  To verify (\ref{ramon}), without loss of generality assume that $p>1$ (otherwise the inequality is trivial). Consider  $g(x)=2|x-1|^{p}-|x|^{p}+1$. Its second derivative changes signs at points $x$ which satisfy the equation $|x-1|=2^{1/(2-p)}|x|$, i.e., when $x=x_{\pm}=\frac{1}{1\pm 2^{1/(2-p)}}$. The right hand side of (\ref{ramon}) represents the tangent line to the graph of $g$ at the point $x=1$. Clearly $g$ is convex on 
$[ x_{+}, \infty)$. Therefore (\ref{ramon}) is true on this interval. Next,  $g$ is concave on $[x_{-}, x_{+}]$ and since $x_{-}<0$, we have $g\geq p(1-x)$  on $[0,x_{+}]$ because $g(0)>p(1-0)$.  Thus (\ref{ramon}) is true for $x\geq 0$. For $x \leq 0$, by Bernoulli we have 
\begin{align*}
2|x-1|^{p}-|x|^{p}+1 - p(1-x) \geq |x-1|^{p}+1 - p(1-x)  \geq 1-px +1-p(1-x)=2-p \geq 0.
\end{align*}

To the best of our knowledge, the constants $c_{p}, T_{p}$ are unknown for $p \in [1,2)$. There is  a remarkable result of Ben-Efraim--Lust-Piquard~\cite{Efraim} that $T_{p}\geq \frac{2}{\pi}$ for $1\leq p\leq 2$. 

This, combined to our theorem,  gives the lower bound $T_{p} \geq \max\{\frac{2}{\pi}, s_{p'}\}$ for $1\leq p \leq 2$.  However, due to the inequalities of Bobkov--G\"otze and   Maurey--Pisier (see the next section),  it is widely believed that $c_{1}=T_{1}=\sqrt{\frac{2}{\pi}}$.  

An elegant  idea of Naor--Schechtman~\cite{NAOR}  based on  Burkholder's inequality~\cite{BUR02} gives an estimate
\begin{align*}
(p-1)^{p}\,  \mathbb{E}|f-\mathbb{E} f|^{p} \leq  \mathbb{E}_{x} \mathbb{E}_{x'} \left|\sum_{j=1}^{n} x'_{j} \partial_{j}f(x) \right|^{p} \quad 1<p\leq 2. 
\end{align*}
Let us show that our bound (\ref{assaf1}) obtained in Corollary~\ref{asebi} is better. 
\begin{prop}\label{comp}
For all $1<p<2$ we have 
\begin{align}\label{gama}
s^{p}_{p'}  2^{\frac{p-2}{2}} \min\left(1,\frac{\Gamma((p+1)/2)}{\Gamma(3/2)} \right) > (p-1)^{p} \quad \text{for} \quad 1<p<2. 
\end{align}
\end{prop}
\begin{proof}
We estimate $s_{p'}$ from below by $\sqrt{2/p'}$ (see Lemma~\ref{smalllest}). Using $\Gamma(3/2)=\sqrt{\pi}/2$ we see that  to prove (\ref{gama}) it is enough to show the following two inequalities 
\begin{align}
&\left( \frac{2}{p'}\right)^{p/2} 2^{\frac{p-2}{2}}> (p-1)^{p} \quad \text{for} \quad 1<p<2; \label{gam1}\\
&\left( \frac{2}{p'}\right)^{p/2} 2^{\frac{p}{2}} \frac{\Gamma((p+1)/2)}{\sqrt{\pi}} >(p-1)^{p}  \quad \text{for} \quad p_{0}<p<2, \label{gam2}
\end{align}
where $p_{0}\approx 1.847...$ is the solution of the equation $\Gamma((p+1)/2)=\sqrt{\pi}/2$ on the interval $(1,2)$.  Inequality (\ref{gam1}) simplifies to  $2^{2-\frac{2}{p}}>p(p-1)$ which is true because the left hand side is concave, and the right hand side is convex on $[1,2]$. To show (\ref{gam2}) it is enough to verify that
$$
4 \left(\frac{\Gamma((p+1)/2)}{\sqrt{\pi}}\right)^{2/p} > p(p-1) \quad \text{for} \quad  1<p<2. 
$$
The latter inequality we rewrite as follows 
\begin{align*}
\frac{p\ln(4)-\ln(\pi)}{2} + \ln\left(\Gamma\left(\frac{p+1}{2}\right)\right) - \frac{p\ln(p(p-1))}{2} >0. 
\end{align*}
Since the Trigamma function is convex
$$
(\ln \Gamma(z))'' = \sum_{n=0}^{\infty} \frac{1}{(z+n)^{2}},
$$
 we estimate $\ln (\Gamma((p+1)/2))$ from below by its tangent line at point $p=2$,  i.e., 
 $$
  \ln\left(\Gamma\left(\frac{p+1}{2}\right)\right) >  \ln (\sqrt{\pi}/2) + (1-\gamma/2 - \ln(2))(p-2) \quad \text{for} \quad p_{0}<p<2, 
 $$
here $\gamma$ is Euler's constant. It is enough to show that 
\begin{align*}
\frac{p\ln(4)-\ln(\pi)}{2} + \ln (\sqrt{\pi}/2) + (1-\gamma/2 - \ln(2))(p-2)  -  \frac{p\ln(p(p-1))}{2}> 0.
\end{align*}
The left hand side is concave  on the interval $(1+\sqrt{2}/2,2)$,  and  at the endpoint cases we have the inequality. Notice that $(1+\sqrt{2}/2)=1.7...<p_{0}=1.82...$, and this finishes the proof. 
\end{proof}

\subsubsection{Gaussian measure on $\mathbb{R}^{n}$}
The application of the Central Limit Theorem to (\ref{nasha}) gives a dimension independent Sobolev inequality.
\begin{corollary} 
For any smooth bounded $f :\mathbb{R}^{n} \to \mathbb{R}$ and any $n\geq 1$, we have 
\begin{align}\label{sobolevp}
s^{p}_{p'}\left(\int_{\mathbb{R}^{n}} |f|^{p} d\gamma - \left|\int_{\mathbb{R}^{n}} f d\gamma\right|^{p} \right) \leq \int_{\mathbb{R}^{n}} |\nabla f|^{p} d\gamma. 
\end{align}
\end{corollary}
The best possible constant in (\ref{sobolevp}), unlike $s_{p'}^{p}$, should not degenerate when $p \to 1+$. Indeed,  (see ~\cite{Ledoux01}, pp. 115)  one has \begin{align}\label{hhu}
\sqrt{\frac{2}{ \pi}} \int_{\mathbb{R}^{n}} \left| f -\int_{\mathbb{R}^{n}} f d\gamma \right| d\gamma \leq \int_{\mathbb{R}^{n}} |\nabla f| d\gamma, 
\end{align}
where the constant $\sqrt{\frac{2}{\pi}}$ is the best possible in the left hand side of (\ref{hhu}). We should  mention that estimate (\ref{hhu}) can be also  easily obtained by a remarkable trick of Maurey--Pisier~\cite{Pisier}.

 
 Notice that (\ref{sobolevp}) cannot be extended for the range of exponents $p>2$ with some positive constant $C(p)$ instead of $s_{p'}^{p}$. Indeed, assume the contrary. Consider $n=1$ and take $f(x) = 1+ax$. Using Jensen's inequality, we obtain 
 \begin{align}\label{Latala}
 (1+a^{2})^{p/2} =\left(\int_{\mathbb{R}}|1+ax|^{2} d\gamma \right)^{p/2}\leq \int_{\mathbb{R}}|1+ax|^{p}d\gamma \stackrel{(\ref{sobolevp})}{\leq} \frac{|a|^{p}}{C(p)} +1.
 \end{align}
 Therefore,  taking $a \to 0$, we obtain the contradiction with $pa^{2}/2>  \frac{|a|^{p}}{C(p)}$ for $p>2$.

\subsection{Discrete surface measure}

Let $A \subset \{-1,1\}^{n}$ be a subset of the Hamming cube with cardinality $|A|=2^{n-1}$. Define $w_{A} : \{-1,1\}^{n} \to \mathbb{N} \cup \{0\}$ so that $w_{A}(x)$ is the number of boundary edges to $A$ containing $x$, i.e., $w_{A}(x)$ counts the number of edges  with one endpoint in $A$ and another one in the complement of $A$ such  that  one of the endpoints  is $x$.  Clearly $w_{A}(x)=0$ if $x$ is in the ``strict interior'' of $A$, or in the ``strict complement'' of $A$, and it is nonzero if and only if $x$  is on the ``boundary'' of $A$. Notice that $w_{A}(x)$ can be nonzero for some $x \notin A$. 
The function $w_{A}$ maybe be understood as a discrete surface measure of the boundary of $A$. 
Consider the following quantity 
\begin{align}\label{Guillame}
\sigma(p) = \inf_{A \subset \{-1,1\}^{n}, \; |A|=2^{n-1}} \mathbb{E} w_{A}^{p/2}(x).
\end{align}
It follows from Harper's edge-isoperimetric inequality~\cite{Harp} that $\sigma(2)=1$ and the value is attained on the halfcube. The monotonicity of $\sigma(p)$ in $p$  implies that $\sigma(p)=1$ for all $p\geq 2$. Also notice that considering Hamming balls,  one can easily show that $\sigma(p)=0$ for $0\leq p <1$. Therefore the first nontrivial value is  $\sigma(1)$.  In this case it follows from Bobkov's inequality~(see \cite{BOB} and references therein) that $\sigma(1) \geq \sqrt{\frac{2}{\pi}} \approx 0.79$, and by monotonicity we obtain that $\sigma(p) \geq \sqrt{\frac{2}{\pi}}$ which is definitely not sharp when $p \to 2-$. 

Define $f : \{-1,1\}^{n} \to \{-1,1\}$ as follows: $f(x)=1$ if $x \in A$ and $f(x)=-1$ if $x \notin A$. Clearly $|\nabla f(x)|^{2} = w_{A}(x)$. Applying (\ref{nasha}) to $f$, we obtain 
\begin{align}\label{surface}
\mathbb{E} w_{A}^{p/2}(x) \geq s^{p}_{p'}.
\end{align}
Inequality (\ref{surface}) gives the lower bound $\sigma(p) \geq s^{p}_{p'}$ which tends to $1$ as $p \to 2-$, but  fails to be sharp when $p\to 1+$. Thus combining this result with Bobkov's inequality we obtain the bound 
\begin{align}\label{super}
1\geq \sigma(p) \geq \max\left\{ \sqrt{\frac{2}{\pi}},\,  s^{p}_{p'}\right\} \quad \text{for} \quad 1 \leq p \leq 2. 
\end{align}

\appendix 
\section{Appendix}
\subsection{Properties of $N_{\alpha}(t)$}
\begin{lemma}\label{root}
For any $\alpha \geq 2$, we have $0<s_{\alpha}\leq 1$. In addition $s_{\alpha}$ is decreasing in $\alpha>0$,   and $N'_{\alpha}(t), N''_{\alpha}(t) \leq 0$ on $[0, s_{\alpha}]$ for $\alpha>0$. 
\end{lemma}
\begin{proof}
Consider $G_{\alpha}(t) :=e^{-t^{2}/4}N_{\alpha}(t)$. Notice that the zeros of $G_{\alpha}$ and $N_{\alpha}$ are the same.  It follows from (\ref{hermit}) that 
\begin{align}\label{hyper1}
G''_{\alpha} + \left(\alpha+\frac{1}{2}-\frac{t^{2}}{4} \right) G_{\alpha}=0, \quad G_{\alpha}(0)=1 \quad \text{and} \quad G'_{\alpha}(0)=0.
\end{align}
Besides we know that the solution is even. Consider the critical case $\alpha=2$. In this case $G_{2}(t)=e^{-t^{2}/4}(1-t^{2})$ and the smallest positive zero is $s_{2} = 1$. Therefore it follows from the Sturm comparison principle that $0<s_{\alpha}<1$ for $\alpha>2$ (see below). Moreover, the same principle applied to $G_{\alpha_{1}}$ and $G_{\alpha_{2}}$  with $\alpha_{1}> \alpha_{2}$ implies that $G_{\alpha_{1}}$ has a zero inside the interval $(-s_{\alpha_{2}}, s_{\alpha_{2}})$. Thus  we conclude that  $s_{\alpha}$ is decreasing in $\alpha$.  

To verify that $N'_{\alpha}, N''_{\alpha} \leq 0$ on $[0,s_{\alpha}]$, first we claim that
\begin{align*}
N_{\alpha_{2}} \geq N_{\alpha_{1}}  \quad \text{on} \quad [0, s_{\alpha_{1}}]
\end{align*}
for $\alpha_{1}>\alpha_{2}>0$. Indeed the proof works in the same way as the proof of Sturm's comparison principle. For the convenience of the reader we decided to include the argument. As before, consider $G_{\alpha_{j}} = e^{-t^{2}/4} N_{\alpha_{j}}$. It is enough to show that $G_{\alpha_{2}} \geq G_{\alpha_{1}}$ on $[0,s_{\alpha_{1}}]$. It follows from (\ref{hyper1}) that $G''_{\alpha_{2}}(0) > G''_{\alpha_{1}}(0)$. Therefore,  using the Taylor series expansion at the point $0$, we see that the claim is true at some neighbourhood of zero, say $[0, \varepsilon)$ with $\varepsilon$ sufficiently small. Next we assume the contrary, i.e.,  that there is a point $a \in [\varepsilon, s_{\alpha_{1}}]$ such that $G_{\alpha_{2}}\geq G_{\alpha_{1}}$ on $[0,a]$,  $G_{\alpha_{2}}(a)=G_{\alpha_{1}}(a)$ and $G'_{\alpha_{2}}(a)<G'_{\alpha_{1}}(a)$ (notice that the case $G'_{\alpha_{2}}(a)=G'_{\alpha_{1}}(a)$,  by the uniqueness theorem for ODEs, would imply that $G_{\alpha_{2}}=G_{\alpha_{1}}$ everywhere, which is impossible). Consider the Wronskian 
\begin{align*}
W=G_{\alpha_{1}}'G_{\alpha_{2}} - G_{\alpha_{1}} G'_{\alpha_{2}}.
\end{align*}
We have  $W(0) = 0$ and $W(a)=G_{\alpha_{1}}(a) (G'_{\alpha_{1}}(a) - G'_{\alpha_{2}}(a)) \geq  0$. On the other hand, we have 
\begin{align*}
W' = (\alpha_{2} -\alpha_{1}) G_{\alpha_{1}}G_{\alpha_{2}} < 0 \quad \text{on} \quad [0,a),
\end{align*}
which is a clear contradiction, and this proves the claim. 

It follows from (\ref{series}) that 
\begin{align}\label{recurent}
N''_{\alpha} = -\alpha N_{\alpha-2},
\end{align}
 and inequalities  $N_{\alpha-2} \geq  N_{\alpha}\geq 0$ on $[0,s_{\alpha}]$ imply that $N''_{\alpha}\leq 0$ on $[0, s_{\alpha}]$. Since $N'_{\alpha}(0)=0$ and $N''_{\alpha}\leq 0$ on $[0, s_{\alpha}]$, we must have $N'_{\alpha}\leq 0$ on $[0,s_{\alpha}]$. 
\end{proof}

\begin{lemma}\label{smalllest}
We have $s_{\alpha} \geq \sqrt{\frac{2}{\alpha}}$ for all $\alpha \geq 2$. 
\end{lemma}
\begin{proof}
Notice that  $G_{2}(x) = e^{-x^{2}/4}(1-x^{2})$ satisfies (\ref{hyper1}) with $\alpha =2$. Now consider $V_{\alpha}(x):=G_{2}(x\sqrt{\alpha/2})$. The function $V_{\alpha}(x)$ satisfies the equation 
\begin{align*}
V''_{\alpha}(x) + \left(\alpha + \frac{\alpha}{4} - \frac{\alpha^{2} x^{2}}{16}  \right)V_{\alpha}(x)=0,
\end{align*}
and $V_{\alpha}(0)=1$, $V'_{\alpha}(0)=0$. Notice that $V_{\alpha}(x)>0$  on $[0, \sqrt{2/\alpha})$. Since 
\begin{align*}
\alpha+ \frac{\alpha}{4} - \frac{\alpha^{2} x^{2}}{16} \geq \alpha+\frac{1}{2} - \frac{x^{2}}{4} \quad \text{for} \quad x \in [0, \sqrt{2/\alpha}),
\end{align*}
it follows from the Sturm comparison principle (see the previous discussions) that $G_{\alpha} >V_{\alpha}>0$ on $(0, \sqrt{2/\alpha})$. Thus we obtain that $s_{\alpha} \geq \sqrt{2/\alpha}$. 
\end{proof}

\subsection{Heat inequality}\label{heatt}
Let $U(p,q)$ be defined as in (\ref{upq}).
\begin{lemma}\label{blema}
For any $p \in \mathbb{R}$, the map 
\begin{align}\label{convsq}
 t \mapsto U(p, \sqrt{t}) \quad \text{for} \quad t\geq 0
\end{align}
is convex. 
\end{lemma}
\begin{proof}
Without loss of generality, assume that $p\geq 0$. We recall that $U(p,\sqrt{t}) = t^{\alpha/2}u_{\alpha}(p/\sqrt{t})$.  Since $\alpha \geq 2$, the only interesting case to consider is when $p/\sqrt{t} <s_{\alpha}$ (otherwise $ t^{\alpha/2}$ is convex). In this case we have $U(p,\sqrt{t}) = t^{\alpha/2} N_{\alpha}(p/\sqrt{t})$ up to a positive constant which we are going to ignore, and, therefore, by (\ref{hermit}) we have  $(U(p,\sqrt{t}))_{t}+\frac{(U(p,\sqrt{t}))_{pp}}{2}=0$. Using (\ref{recurent}), we obtain
\begin{align*}
(U(p,\sqrt{t}))_{t} = -\frac{(U(p,\sqrt{t}))_{pp}}{2} = -\frac{1}{2}t^{\frac{\alpha}{2}-1}N''_{\alpha}(p/\sqrt{t}) = \frac{\alpha}{2}t^{\frac{\alpha-2}{2}}N_{\alpha-2}(p/\sqrt{t}).
\end{align*}
Therefore it would be enough to show that for any $\gamma \geq 0$,  the function   $\frac{N_{\gamma}(x)}{x^{\gamma}}$
is decreasing  for $x \in (0,s_{\gamma+2})$. 
Differentiating,  and using (\ref{hermit}) again, we obtain 
\begin{align*}
\frac{\mathrm{d}}{\mathrm{d}x} \, \left(\frac{N_{\gamma}(x)}{x^{\gamma}}\right) =\frac{N''_{\gamma}(x)}{x^{\gamma+1}},
\end{align*}
which  is nonpositive by Lemma~\ref{root}. 
\end{proof}

The next lemma, together with Lemma~\ref{blema} and (\ref{infinitesimal}), implies that $U(p,q)$ satisfies (\ref{neravenstvo}). 

\begin{lemma}[Barthe--Mauery~\cite{BM}]\label{barko}
Let $J$ be a convex subset of $\mathbb{R}$, and let $V(p,q) : J \times \mathbb{R}_{+} \to \mathbb{R}$ be such that 
\begin{align}
&V_{pp}+\frac{V_{q}}{q} \leq 0 \quad \text{for all}\quad (p,q) \in J\times \mathbb{R}_{+}; \label{inffed}\\
&t \mapsto V(p, \sqrt{t}) \quad \text{is convex for each fixed} \quad p \in J.\label{amoz}
\end{align}
Then for all  $(p,q,a)$ with  $p\pm a \in J$ and $q\geq 0$, we have 
\begin{align}
2V(p,q)\geq V(p+a, \sqrt{a^{2}+q^{2}})+V(p-a,\sqrt{a^{2}+q^{2}}). \label{toloba}
\end{align}
\end{lemma}
The lemma says that the global discrete inequality (\ref{toloba}) is in fact implied by its infinitesimal form (\ref{inffed}) under the extra condition (\ref{amoz}).
\begin{proof}
The argument is borrowed from \cite{BM}. The similar argument was used by Davis~\cite{Davis} in obtaining sharp square function estimates from 
the ones for the  Brownian motion. 

Without loss of generality assume $a\geq 0$. Consider the process 
\begin{align*}
X_{t} = V(p+B_{t}, \sqrt{q^{2}+ t}),  \quad t\geq 0.
\end{align*}
Here $B_{t}$ is the standard Brownian motion starting at zero. It follows from Ito's formula  together with (\ref{inffed}) that $X_{t}$ is a supermartingale. 
Let  $\tau$ be the stopping time
\begin{align*}
\tau = \inf\{ t \geq 0 : B_{t} \notin (-a,a)\}.
\end{align*} 

It follows from the optional stopping theorem that 
\begin{align*}
&V(p,q) = X_{0} \geq \mathbb{E} X_{\tau}=\mathbb{E} V(p+B_{\tau}, \sqrt{q^{2}+\tau}) = \\
&P(B_{\tau}=-a) \mathbb{E}(V(p-a, \sqrt{q^{2}+\tau}) | B_{\tau}=-a)+
P(B_{\tau}=a) \mathbb{E}(V(p+a, \sqrt{q^{2}+\tau}) | B_{\tau}=a)=\\
&\frac{1}{2}\left(\mathbb{E}(V(p-a, \sqrt{q^{2}+\tau}) | B_{\tau}=-a)+
\mathbb{E}(V(p+a, \sqrt{q^{2}+\tau}) | B_{\tau}=a) \right) \geq \\
&\frac{1}{2}\left(V\left( p-a, \sqrt{q^{2}+ \mathbb{E}(\tau|B_{\tau}=-a) }\right)+V\left( p+a, \sqrt{q^{2}+ \mathbb{E}(\tau|B_{\tau}=a) }\right) \right)=\\
&\frac{1}{2}\left( V\left(p-a, \sqrt{q^{2}+a^{2}}\right) +  V\left(p+a, \sqrt{q^{2}+a^{2}}\right) \right)\,.
\end{align*}

Notice that  we have used  $P(B_{\tau}=a)=P(B_{\tau}=-a)=1/2$, $\mathbb{E}(\tau | B_{\tau}=a) =\mathbb{E}(\tau | B_{\tau}=-a)=a^{2}$, and the fact that the map $t \mapsto V(p, \sqrt{t})$ is convex together with Jensen's inequality.  
\end{proof}

\subsection{Minimax theorem for noncompact sets}\label{minnn}
Let $P, Q$ be nonempty closed  convex sets in $\mathbb{R}$. 
We say that a pair $(p^{*}, q^{*}) \in P\times Q$ is a saddle point of $f$ on $P\times Q$ if 
\begin{align*}
f(p,q^{*})\leq f(p^{*}, q^{*}) \leq f(p^{*}, q) \quad \text{for all} \quad (p,q) \in P\times Q.
\end{align*}

\begin{lemma}\label{magari2}
The function $f$ defined on $P\times Q$ with real values possesses a saddle point $(p^{*}, q^{*})$ on $P\times Q$ if and only if 
\begin{align*}
\max_{p\in P}\inf_{q \in Q} f(p,q)  = \min_{q\in Q}\sup_{p \in P} f(p,q),
\end{align*}
and this number is then equal to $f(p^{*}, q^{*})$. 
\end{lemma}
For the proof we refer the reader to Proposition~1.2 in \cite{ET}, pp. 167.

\begin{theorem}\label{hartung}
Suppose that $f : P\times Q \to \mathbb{R}$ is continuous, concave  in $p$,   convex  in $q$, and there exists $(p_{0}, q_{0})\in P\times Q$ such that 
\begin{align*}
\lim_{p \in P, \; |p| \to \infty} f(p,q_{0})=-\infty \quad \text{and} \quad \lim_{q\in Q,\; |q| \to \infty} f(p_{0}, q)=+\infty. 
\end{align*}
Then $f$ possesses at least one saddle point on $P\times Q$ and 
\begin{align*}
f(p^{*},q^{*})=\min_{q\in Q}\sup_{p \in P} f(p,q)=\max_{p\in P}\inf_{q \in Q} f(p,q).
\end{align*}
\end{theorem}
The theorem is  Proposition~2.2 in \cite{ET}, pp. 173.

\section*{Acknowledgments} 
We are very grateful to several people for discussions and suggestions that led us to noticing  the duality between the Hamming cube and the square function: G.~Aubrun for valuable remarks on optimizers in (\ref{Guillame});   D.~Bilyk for providing the reference to sharp constants for Square functions~\cite{Wang}; R.~O'Donell for providing the references; R.~Lata\l a for pointing out example (\ref{Latala}); S.~Petermichl for bringing our attention to  Bellman functions in Square function estimates and Poincar\'e inequalities for the Gaussian measure; S.~Treil for attracting our attention to Chang--Wilson--Wolff's superexponential bound (Corollary~\ref{Wolff})   and its similarity to the Gaussian concentration inequality; R.~van Handel for references, including (\ref{piquard1}) and (\ref{hhu}),  and making several important remarks. We thank an anonymous referee for helpful comments and remarks. 

\bibliographystyle{amsplain}



\begin{dajauthors}
\begin{authorinfo}
  Paata Ivanisvili\\
  Mathematics Department\\
  Princeton University\\
  Princeton, NJ 08544\\
  and\\
  Department of Mathematics\\
  University of California, Irvine\\
  Irvine, CA 92697-3875\\
  paatai\imageat{}princeton\imagedot{}edu \\
\end{authorinfo}
\begin{authorinfo}
  Fedor Nazarov\\
  Department of Mathematics\\
  Kent State University\\
  Kent, OH 44240\\
  nazarov\imageat{}math\imagedot{}kent\imagedot{}edu \\
\end{authorinfo}
\begin{authorinfo}
  Alexander Volberg\\
  Department of Mathematics\\
  Michigan State University\\
  East Lansing, MI 48823\\
  volberg\imageat{}math\imagedot{}msu\imagedot{}edu\\
\end{authorinfo}
\end{dajauthors}

\end{document}